\documentclass{article} 
\usepackage{geometry}
\usepackage{amsmath}
\usepackage{amsthm}
\usepackage{amsfonts}
\usepackage{mathrsfs}
\usepackage{stmaryrd}
\usepackage{setspace}
\usepackage{fullpage}
\usepackage{amssymb}
\usepackage{breqn}
\usepackage{enumitem}
\usepackage{authblk}
\usepackage{comment}
\usepackage{hyperref}

\newtheorem*{rep@theorem}{\rep@title}
\newcommand{\newreptheorem}[2]{%
\newenvironment{rep#1}[1]{%
 \def\rep@title{#2 \ref{##1}}%
 \begin{rep@theorem}}%
 {\end{rep@theorem}}}
\makeatother

\theoremstyle{plain}
\newtheorem{theorem}{Theorem}[section]
\newtheorem{thm}{Theorem}[section]

\newtheorem{prop}[theorem]{Proposition}
\newreptheorem{theorem}{Theorem}
\newtheorem{lemma}[theorem]{Lemma}

\theoremstyle{definition}

\newtheorem{defn}[theorem]{Definition}

\newtheorem{clm}[theorem]{Claim}

\newtheorem*{rem*}{Remark}
\newtheorem*{proposition*}{Proposition}

\newcommand\ex{\ensuremath{\mathrm{ex}}}
\newcommand\cA{{\mathcal A}}
\newcommand\cF{{\mathcal F}}
\newcommand\cN{{\mathcal N}}
\newcommand\cQ{{\mathcal Q}}

\title{On Turán-type problems and the abstract chromatic number}

\author{
  \begin{minipage}[t]{0.3\linewidth}
    \centering
    D\'aniel Gerbner\thanks{HUN-REN Alfr\'ed R\'enyi Institute of Mathematics, E-mail: \texttt{gerbner@renyi.hu}. Research supported by the National Research, Development and Innovation Office - NKFIH under the grants FK 132060 and KKP-133819.}
  \end{minipage}
  \begin{minipage}[t]{0.3\linewidth}
    \centering
    Hilal Hama Karim\thanks{Department of Computer Science and Information Theory, Faculty of Electrical Engineering and Informatics, Budapest University of Technology and Economics, Műegyetem rkp. 3., H-1111 Budapest, Hungary. E-mail: \texttt{hilal.hamakarim@edu.bme.hu}. Research supported by the National Research, Development and Innovation Office - NKFIH under the grant FK 132060.}
  \end{minipage}
  \begin{minipage}[t]{0.3\linewidth}
    \centering
    Gaurav Kucheriya\thanks{Department of Applied Mathematics, Charles University, Prague, Czechia, Email: \texttt{gaurav@kam.mff.cuni.cz}. Supported by GA\v{C}R grant 22-19073S and SVV–2023–260699.}
  \end{minipage}
}
\date{}

\begin{document}
\maketitle

\begin{abstract}
    In 2020, Coregliano and Razborov introduced a general framework to study limits of combinatorial objects, using logic and model theory. They introduced the abstract chromatic number and proved/reproved multiple Erdős-Stone-Simonovits-type theorems in different settings. In 2022, Coregliano extended this by showing that similar results hold when we count copies of $K_t$ instead of edges.

    Our aim is threefold. First, we provide a purely combinatorial approach. Second, we extend their results by showing several other graph parameters and other settings where Erdős-Stone-Simonovits-type theorems follow. Third, we go beyond determining asymptotics and obtain corresponding stability, supersaturation, and sometimes even exact results.
\end{abstract}

\section{Introduction}

One of the most fundamental results in extremal combinatorics is the theorem of Turán \cite{T1941}, which determines the maximum number of edges among $n$-vertex graphs that do not contain $K_{k+1}$ as a subgraph, in other words, $K_{k+1}$-free graphs. More generally, given a graph $F$, let $\ex(n,F)$ denote the largest $|E(G)|$ among $n$-vertex $F$-free graphs $G$. Tur\'an's theorem \cite{T1941} states that $\ex(n,K_{k+1})=|E(T(n,k))|$, where $T(n,k)$ is the complete $k$-partite graph with each part of order $\lfloor n/k\rfloor$ or $\lceil n/k\rceil$. The celebrated Erd\H{o}s-Stone-Simonovits (ESS) theorem \cite{ES1966,ES1946} is the most general result in the area, which states that the same holds if we forbid another graph with chromatic number $k+1$, apart from an error term $o(n^2)$, i.e., for any graph $F$ we have $\ex(n,F)=|E(T(n,\chi(F)-1))|+o(n^2)$. Note that this determines the asymptotics of $\ex(n,F)$, if $F$ is not bipartite.

There are several different generalizations of the Turán problem. One line of research is to consider graphs with some extra structure. Various examples include vertex-ordered \cite{pachtar,tard}, cyclically ordered \cite{bkv} and edge-ordered graphs \cite{GMNPTV}. Probably the main common theme in all of these studies is obtaining an analogous result to the Erd\H{o}s-Stone-Simonovits theorem.  The essence of this is to find an appropriate notion of ``chromatic number" in those contexts so that it can play the role of the usual chromatic number in the ESS result. For instance, we have the \textit{interval chromatic number} for the vertex-ordered graphs and the \textit{order chromatic number} for the edge-ordered graphs.

There have been attempts to study these types of problems in a unified and general way. Coregliano and Razborov \cite{CoregRazb} introduced a general model theoretic framework  to study limits of combinatorial objects. They defined \textit{abstract chromatic number} of ``open interpretations" on theories of graphs to capture such different notions of chromatic numbers in a unified way. They obtained the ESS result for the density of the edges in this general setting and Coregliano \cite{Coreg} extended this to the density of cliques. 

Both in the ad hoc manner and in the unified approach, the general aim is to determine the objective extremum (e.g. maximum number of edges or cliques) among the set of all graphs that can be underling graphs of the graphs with the extra structure that possess some desired properties (e.g. not containing certain forbidden configurations).  For instance, in vertex-ordered graphs, we are interested in the maximum number of edges of an $n$-vertex vertex-ordered graph $G$ avoiding $F$ in an ordered sense. The ordering does not play any role in counting the number of edges, thus we can think of this as counting the number of edges of the \textit{underlying graph} of $G$, i.e., the ordinary graph we obtain from $G$ by simply ignoring the ordering. The problem then reduces to finding the largest number of edges among graphs that can be underlying graphs of $F$-free vertex-ordered graphs. This way, the family of all graphs is partitioned into a family $\cA(F)$ of \textit{allowed graphs} and a family $\cF(F)$ of forbidden graphs. In each case of graphs with an extra structure, and in the model theoretic approach, the corresponding chromatic number is defined in a way that if its value for a graph $F$ is $k$, then we have that the Tur\'an graph $T(n,k-1)$ is among $\cA(F)$ and for $n$ large enough $T(n,k)$ is in $\cF(F)$. This is the core idea in the proofs for the ESS-like results.

In this paper we introduce a general, unified and yet purely combinatorial approach. We consider partitions $(\cA,\cF)$ of the family of all graphs into $\cA$ and $\cF$, and define the ``abstract chromatic number" of such partitions. Let $K_k(n)$ denote $T(nk,k)$ and let $T(n,\infty)=K_n$. 

\begin{defn}
We say that a partition $(\cA,\cF)$ is \textit{Turán-suitable} 
if there is a $k$ such that for sufficiently large $n$, each complete $(k-1)$-partite graph with each part of order at least $n$ is in $\cA$ but no $G\in\cA$ contains $T(n,k)$ as a subgraph, or if $K_n\in\cA$ for sufficiently large $n$. For simplicity, we will say \textit{suitable} instead of Turán-suitable for the rest of this paper. 
\end{defn}

Let us see some examples of suitable partitions. We say that a partition is \textit{monotone} if $G\in\cA$ implies that every subgraph of $G$ is in $\cA$. This is clearly the case when a graph with extra structure is forbidden. We say that the partition is \textit{hereditary} if $G \in \cA$ implies that every induced subgraph of $G$ is in $\cA$. This is the case when some graphs are forbidden as induced subgraphs. Clearly monotone partitions are hereditary.

\begin{lemma}\label{lemi}
    Hereditary partitions are suitable.
\end{lemma}

\begin{proof}  If $K_{n}\in \cA$ for every $n$ sufficiently large, then we are done. Otherwise, there exists a maximum value of $k$ such that for sufficiently large $n$, $T(n,k-1)\in \cA$. 

Then $T(m,k)\not\in \cA$ for every sufficiently large $m$. Then for $G\in\cA$, $T(m,k)$ is not an induced subgraph of $G$ by the hereditary property.  We pick $m$ such that $K_m\not\in\cA$. Let $n$ be large enough with respect to $m$. Assume that $G$ contains a $T(n,k)$ as a (not necessarily induced) subgraph and let $U_1,\dots,U_k$ be the partite sets of this $T(n,k)$.
If each $U_i$ contains an independent set of order at least $m$, we found an induced $T(m,k)$. Otherwise, one of the parts, say $U_1$ does not contain an independent set of order $m$, thus it contains $K_m$ by Ramsey's theorem, a contradiction.
\end{proof}

A \textit{graph with extra structure} is given by a pair $(G,X)$ where $G$ is a graph, and $X$ represents some extra structure, like orderings of vertices or edges, or coloring, and so on. Let $\cQ$ denote a family of pairs $(G,X)$. Assume that we are given a transitive relation $<$ on the pairs in $\cQ$ such that if $(G,X)<(G',X')$, then $G$ is a subgraph of $G'$. Assume furthermore that for every $(G,X)$ and every subgraph $G'$ of $G$, there is $X'$ such that $(G',X')<(G,X)$. Note that in our examples of extra structures, $X'$ can be the restriction of $X$ to $G'$.

Given a family $\cF_0\subset \cQ$ we say that a pair $(G,X)$ is $\cF_0$-free if there is no $(F,Y)\in \cF_0$ with $(F,Y)<(G,X)$. For $\cF_0$ the \textit{corresponding partition} $(\cA,\cF)$ of graphs is defined in the following way. We have $G\in \cA$ if there is an $X$ such that $(G,X)$ is $\cF_0$-free. Then this partition is monotone, since for each subgraph $G'$ of $G\in\cA$, we have an $X'$ such that $(G',X')<(G,X)$. If $G'\not\in\cA$, then there is $(F,Y)\in \cF_0$ with $(F,Y)<(G',X')$. Therefore, $(F,Y)<(G,X)$ by the transitivity of $<$, hence $G\not\in\cA$, a contradiction.

Next we define the abstract chromatic number of a suitable partition, which coincides with the definition given in \cite{CoregRazb} in the simple specific cases they present as examples. 

\begin{defn}
    Given a suitable partition, its \textit{abstract chromatic number} is $\infty$ if $K_n\in\cA$ for sufficiently large $n$. Otherwise, the abstract chromatic number is the largest $k$ such that every complete $(k-1)$-partite graph with each part of order at least $n$ is in $\cA$, for every sufficiently large $n$.
\end{defn}

Note that $k$ is the same as in the definition of suitable partitions. When  $(\cA,\cF)$ is the corresponding partition of some family $\cF_0$ of graphs with extra structure, then we say that the abstract chromatic number of $\cF_0$ is the abstract chromatic number of $(\cA,\cF)$.

Also note that our approach is in some sense stronger than that of \cite{CoregRazb}. They deal only with finitely axiomatizable theories (although mention that it is easy to extend their results). For example, the case that $\cA$ is the family of bipartite graphs does not fit into their setting but is handled by our approach.

\smallskip

Let us consider some graph parameter $h(G)$. Let $$g(n,F)=g_h(n,F)=\max\{h(G): G \text{ is an $n$-vertex $F$-free graph}\}.$$ Then we say that $g$ is a \textit{Tur\'an-type function}. For instance, in the classical Tur\'an problem $h(G)=|E(G)|$. For other examples, see Section \ref{examp}. 

We extend this to suitable partitions as follows. $$g(n,(\cA,\cF)):=\max\{h(G): G \text{ is an $n$-vertex graph in $\cA$}\}.$$ Note that if $(\cA,\cF)$ is a monotone partition, then $g(n,(\cA,\cF))$ is simply $g(n,\cF)$.

\smallskip
As we mentioned above, an essential result in those generalizations of the Tur\'an problem is the ESS theorem. Therefore, we define the notion of $k$-ESS for the Tur\'an type functions.

\begin{defn}
We say that $g$ is \textit{weakly $k$-ESS} if for any graph $F$ with chromatic number $k$, $g(n,F)=(1+o(1))h(T)$ for an $n$-vertex complete $(k-1)$-partite graph $T$. We say that $g$ is \textit{strongly $k$-ESS} if the above holds with $T$ being the Tur\'an graph $T(n,k-1)$.
\end{defn}

We say that $g$ is \textit{weakly $k$-ESS with respect to a partition $(\cA,\cF)$} if $g(n,(\cA,\cF))=(1+o(1))h(T)$ for a complete $(k-1)$-partite graph $T$ on $n$ vertices. We say that $g$ is \textit{strongly $k$-ESS with respect to $(\cA,\cF)$} if the above holds with $T$ being the Tur\'an graph $T(n,k-1)$.

\begin{thm}\label{main}
If $g$ is a weakly (resp. strongly) $k$-ESS Turán-type function, then $g$ is also weakly (resp. strongly) $k$-ESS with respect to any suitable partition with abstract chromatic number $k$. 
\end{thm}

Note that if the abstract chromatic number of $(\cA,\cF)$ is infinity, then clearly $g(n,(\cA,\cF))=h(K_n)$.

\begin{proof} 
    Let $n$ be sufficiently large and $n'\ge kn$.
    We have that $T(n',k-1)\in\cA$, by the definition of the abstract chromatic number. Let $T$ be such that $h(T)\ge (1+o(1))h(T')$ for every $T'$ such that $T$ and $T'$ are both $n'$-vertex complete $(k-1)$-partite graphs with each part of order at least $n$. Then $T\in \cA$ because the partition is suitable. This gives the lower bound. 

    We also have that $T(m,k)\in\cF$ and $K_m\in\cF$ for some $m$ by the definition of the abstract chromatic number. Let $n$ be large enough and $G$ be an $n$-vertex graph in $\cA$. Then we have that $G$ is $T(m,k)$-free by the definition of suitable partitions. Together with the weak (or strong) $k$-ESS property, we obtain the upper bound.
\end{proof}

The rest of the paper is organized as follows. In Section \ref{sec2} we define some Tur\'an type properties and show that such properties imply analogous results with respect to suitable partitions. In Section \ref{examp} we list several Tur\'an-type functions that are known to be weakly or strongly $k$-ESS. We also present some 
examples for the various properties studied in Section \ref{sec2}. Afterwards, we present some suitable partitions. In Section \ref{harom} we show how some natural restriction on $h$ allows us to obtain further exact results. We finish the paper with some concluding remarks in Section \ref{negy}.

\section{Stability, supersaturation and more}\label{sec2}

We say that a Tur\'an-type function $g=g_h$ is \textit{weakly $k$-ESS-stable}, if $g$ is weakly $k$-ESS and for any graph $F$ with chromatic number $k$, any $F$-free $n$-vertex $G$ with $h(G)\ge (1-o(1)) g(n,F)$ can be turned to an $n$-vertex complete $(k-1)$-partite graph $T$ by adding and/or deleting $o(n^2)$ edges. 
We say that $g$ is \textit{strongly $k$-ESS-stable} if the above holds with $T$ being the Tur\'an graph $T(n,k-1)$ and $g$ is strongly $k$-ESS. 

We say that $g$ is \textit{weakly $k$-ESS-stable with respect to a partition $(\cA,\cF)$} if the following holds. If $G$ is an $n$-vertex graph in $\cA$ with $h(G)\ge (1-o(1))g(n,(\cA,\cF))$, then $G$ can be turned to an $n$-vertex complete $(k-1)$-partite graph $T$ by adding and/or deleting $o(n^2)$ edges. We say that $g$ is \textit{strongly $k$-ESS-stable with respect to $(\cA,\cF)$} if the above holds with $T$ being the Tur\'an graph $T(n,k-1)$.

\begin{thm}   If $g$ is a weakly (resp. strongly) $k$-ESS-stable Turán-type function, then $g$ is also weakly (resp. strongly) $k$-ESS-stable with respect to any suitable partition with abstract chromatic number $k$.\end{thm}

\begin{proof}
Let $T$ be a complete $(k-1)$-partite graph with $g(n,(\cA,\cF))=(1+o(1))h(T)$, which exists by Theorem \ref{main}. Let $G$ be an $n$-vertex graph in $\cA$ with $h(G)\ge (1-o(1))g(n,(\cA,\cF))$. We have that for large enough $m$, $T(m,k)\in\cF$ and $G$ is $T(m,k)$-free by definition of the abstract chromatic number. We claim that $g(n,T(m,k))=(1+o(1))h(T)$. Indeed, this holds for some $n$-vertex complete $(k-1)$-partite graph $T'$ because $g$ is $k$-ESS, and we must have $h(T')\le (1+o(1))h(T)$, since $g(n,(\cA,\cF))=(1+o(1))h(T)$.

Now we have $h(G)\ge (1-o(1))h(T)=(1-o(1))g(n,T(m,k))$ and we are done since $g$ is weakly $k$-ESS-stable. The strong case follows similarly.
\end{proof}

Given a graph $F$ of chromatic number $k$, we let $\sigma(F)$ denote the smallest color class among all possible proper $k$-colorings of $F$. Given a family $\cF$ of graphs with smallest chromatic number $k$, we let $\sigma(\cF)$ be the smallest $\sigma(F)$ among $k$-chromatic elements of $\cF$.

We say that a Tur\'an-type property $g=g_h$ is \textit{weakly $k$-ESS-sigma} if $g(n,F)=h(T)$ for some $n$-vertex complete $(k+\sigma(F)-1)$-partite graph $T$ with $\sigma(F)-1$ parts of order $1$. In other words, we obtain $T$ from an $(n-\sigma(F)+1)$-vertex complete $(k-1)$-partite graph by adding $\sigma(F)-1$ vertices and joining each of them to every other vertex. Let $T(n,k-1,t)$ be the graph we obtain from $T(n-t,k-1)$ by adding $t$ vertices and joining each of them to every other vertex.

We say that $g$ is \textit{strongly $k$-ESS-sigma} if the above holds with $T$ being $T(n,k-1,\sigma(F)-1)$.
We remark that we know of only one example of such functions, counting $T((k-1)a,a)$ for $a$ large enough, see Section \ref{examp} for more details.

Given a suitable partition $(\cA,\cF)$ with abstract chromatic number $k<\infty$, we let $\sigma(\cA,\cF)$ be the smallest $\sigma(T)$ for complete $k$-partite graphs $T$ such that no element of $\cA$ contains $T$ as a subgraph. Note that if the partition is monotone, then $\sigma(\cA,\cF)=\sigma(\cF)$. 

We say that a Tur\'an-type function $g=g_h$ is \textit{weakly $k$-ESS-sigma with respect to a partition $(\cA,\cF)$} if $g(n,(\cA,\cF))=h(T)$ for some $n$-vertex complete $k$-partite graph $T$ with $\sigma(\cA,\cF)-1$ parts of order $1$. We say that  $g$ is \textit{strongly $k$-ESS-sigma with respect to $(\cA,\cF)$} if the above holds with $T$ being the Tur\'an graph $T(n,k-1)$.

\begin{thm}
    If $g$ is a weakly (resp. strongly) $k$-ESS-sigma Tur\'an-type function, then $g$ is also weakly (resp. strongly) $k$-ESS-sigma with respect to any  suitable partition with abstract chromatic number $k$.
\end{thm}

We remark that it is quite shocking to obtain an exact bound here. For example, in the case of edge-ordered graphs and counting edges, the only exact results are the trivial cases with infinite abstract chromatic number, the other trivial cases of stars and triangles where there is only one edge-ordering, and the simplest remaining cases, the two edge-orderings of the 3-edge path. Now we obtain an exact result for every edge-ordered graph, provided we can calculate the abstract chromatic number and $\sigma$. Note that both of these tasks seem very complicated.

\begin{proof}
    By the definition of $\sigma(\cA,\cF)$ we have that $T\in\cA$, giving the lower bound. For the upper bound, recall that there is a complete $k$-partite graph $F\in\cF$ with $\chi(F)=k$, $\sigma(F)=\sigma(\cA,\cF)$ such that no element of $\cA$ contains $F$ as a subgraph. Then we clearly have $g(n,(\cA,\cF))\le g(n,F)$. Using the weakly $k$-ESS-sigma property of $g$ as a Tur\'an-type function completes the proof. The strong version follows similarly.
\end{proof}

\smallskip

We say that a Tur\'an-type function $g=g_h$ is \textit{$k$-ESS-supersat} if for any $\varepsilon>0$ there is $\delta>0$ such that for any sufficiently large $n$, any $n$-vertex graph $G$ with $h(G)>(1+\varepsilon)g(n,F)$ we have that $G$ contains at least $\delta n^{|V(F)|}$ copies of $F$, for any $k$-chromatic graph $F$.

We say that a Tur\'an-type function $g=g_h$ is \textit{$k$-ESS-supersat with respect to a partition $(\cA,\cF)$} if for any $\varepsilon>0$ there is $\delta>0$ such that for any sufficiently large $n$, any $n$-vertex graph $G$ with $h(G)>(1+\varepsilon)g(n,(\cA, \cF))$, there is an $F\in\cF$ such that $G$ contains at least $\delta n^{|V(F)|}$ copies of $F$.

\begin{thm}
    If $g$ is a $k$-ESS-supersat Tur\'an-type function, then $g$ is also $k$-ESS-supersat with respect to any suitable partition with abstract chromatic number $k$. 
\end{thm}

\begin{proof}
Let $G$ be an $n$-vertex graph with $h(G) \geq (1+\varepsilon)g(n,(\cA,\cF))$, where $n$ is sufficiently large. Since the abstract chromatic number of $(\cA,\cF)$ is $k$, then $T(m,k) \in \cF$, for some $m$, and every graph in $\cA$ is $T(m,k)$-free. Thus, $g(n,(\cA,\cF)) \geq g(n,T(m,k))$, and hence $h(G) \geq (1+\varepsilon) g(n, T(m,k))$. Thus, $G$ contains at least $\delta n^{m}$ copies of $T(m,k)$ by the $k$-ESS-supersat property.
\end{proof}

Assume now that $(\cA,\cF)$ corresponds to a graph $(F,Y)$ with extra structure. The above theorem does not say anything about the number of copies of $(F,Y)$ in $(G,X)$ since the number of copies of $(F,Y)$ in $(G,X)$ is not well-defined. In all the examples of extra structures we consider, the extra structure is defined using the vertices and/or edges of the graph, thus the extra structure $X$ in $G$ creates some extra structure on the subgraphs of $G$, simply by restricting $(G,X)$ to the subgraph. This is what we try to capture formally in the next definition.

Assume that we have an equivalence relation on the pairs in $\cQ$ such that $(G,X)\equiv (G',X')$ implies that $G$ is isomorphic to $G'$. In our examples, the isomorphism is extended to keep the extra structure, e.g., if $u$ is before $v$ in a vertex ordering, then the same holds for their images. Given a graph $G$ and $U\subset V(G)$, $G|_U$ denotes the \textit{restriction} of $G$ to $U$, i.e., the graph with vertex set $U$ where for $u,v\in U$, $uv$ is an edge of $G|_U$ if and only if $uv\in E(G)$.

We say that $\cQ$ and $<$ are \textit{ordinary} if for every $(G,X)\in\cQ$ we have a function $f$ from the power set of $V(G)$ to the extra structures such that for each $U\subset V(G)$ we have that $(G|_U,f(U))\in \cQ$ and $(F,Y)<(G,X)$ if and only if there is a $U\subset V(G)$ such that $(F,Y)\equiv(G|_U,f(U))$. Then a copy of $(F,Y)$ in $(G,X)$ is a subgraph of $G$ with vertex set $U$ such that $(F,Y)\equiv(G|_U,f(U))$.

\begin{thm}
    Let $g=g_h$ be a $k$-ESS-supersat Tur\'an-type function and $\cQ$ be a family of graphs with extra structure, $<$ be a relation on $\cQ$ such that $\cQ$ and $<$ are ordinary. Let $(F,Y)\in\cQ$, $(\cA,\cF)$ be the corresponding partition and $k$ be the abstract chromatic number of $(\cA,\cF)$. Let $G$ be an $n$-vertex graph with $h(G)>(1+\varepsilon)g(n,(\cA,\cF))$ and $(G,X)\in\cQ$. Then $(G,X)$ contains at least $\delta n^{|V(F)|}$ copies of $(F,Y)$.
\end{thm}

\begin{proof}
    Similar to the previous proof, we have that $G$ contains at least $\delta n^{m}$ copies of $T(m,k)$. By the ordinary property, each copy $T$ of $T(m,k)$ has an extra structure $X'=f(V(T))$, and $(T,X')$ contains a copy of $(F,Y)$. Clearly less than $\delta n^{|V(F)|}$ copies of $F$ are contained in less than $\delta n^{|V(F)|}\binom{n-|V(F)|}{m-|V(F)|}<\delta n^{m}$ copies of $T(m,k)$, a contradiction completing the proof.
\end{proof}

\section{Tur\'an-type functions and suitable partitions}\label{examp}

Let us list some Tur\'an-type functions that satisfy the requirements of some of our theorems. We start with $k$-ESS functions.

\textbf{Counting edges, cliques.} The examples in \cite{Coreg} and \cite{CoregRazb}. The Erd\H os-Stone-Simonovits theorem \cite{ES1966,ES1946} itself shows that counting edges is strongly $k$-ESS, and a theorem of Alon and Shikhelman \cite{alon} shows that counting $K_t$ is strongly $k$-ESS if $k>t$.

\textbf{Counting asymptotically (weakly) Tur\'an-good graphs.} Considering that the Tur\'an graph is the extremal graph when we forbid cliques and count edges, it is natural to ask: What graphs $H$ have the property that $\ex(n,H,K_{k})=\cN(H,T(n,k-1))$, at least for $n$ large enough? This property was named as \textit{$k$-Tur\'an-good} in \cite{gerpal}. The graph $H$ is \textit{weakly $k$-Tur\'an-good} if $\ex(n,H,K_{k})=\cN(H,T)$ for some complete $(k-1)$-partite graph. Recall that $\cN(H,T)$ denotes the number of copies of $H$ in $G$. We only need an asymptotic version, but we need it to hold for every $k$-chromatic graph $F$ in place of $K_k$. However, this requirement follows from the $k$-Tur\'an-goodness using the removal lemma \cite{efr}: if we have an $n$-vertex $F$-free graph, we can remove each copy of $K_k$ by deleting $o(n^2)$ edges. This way we removed $o(n^{|V(H)|})=o(\ex(n,H,F))$ copies of $H$. 

Asymptotic (and usually exact) $k$-Tur\'an-goodness has been proved for several graphs. Most usually it is accompanied with a stability result that we will return to shortly. Highlights include the following results: complete $t$-partite graphs with $t<k$ are weakly $k$-Tur\'an-good \cite{gyps}, paths are $k$-Tur\'an-good \cite{gerpal}, and each graph is $k$-Tur\'an-good if $k$ is large enough \cite{mnnrw}.

\textbf{Functions of degree sequences.}  Let $f$ be a non-decreasing log-continuous function and $h(G):=\sum_{v\in V}f(d(v))$. Here log-continuous means that for every $\varepsilon>0$ there exists $\delta>0$ such that for any $m,n$ with $m\le n\le (1+\delta)m$ we have $f(m)\le (1+\varepsilon)f(n)$.
Pikhurko and Taraz \cite{pita} showed that $g_h$ is weakly $k$-ESS. The study of the special case $f(n)=n^r$, $r$ is an integer was initiated by Caro and Yuster \cite{cy}, who conjectured that this function is weakly $k$-ESS. It was proved for any real $r\ge1$ by Bollob\'as and Nikiforov \cite{bolnik}. They showed in \cite{bolnik2} that if $r\le k$, then this function is strongly $k$-ESS.

\textbf{Some topological indices.} 
There are several topological indices of the form $h(G)=\sum_{uv\in E(G)}f(d(u),d(v))$. They are used in chemical graph theory. Gerbner \cite{ger1} showed that is $f$ is a monotone increasing polynomial, then $g_h$ is weakly 3-ESS, moreover, weakly 3-ESS-stable.

\textbf{Spectral radius.} Let $h(G)$ denote the spectral radius of the adjacency matrix of $G$. Nikiforov \cite{niki} showed that $g_h$ is strongly $k$-ESS.

\textbf{$p$-spectral radius.} Kang and Nikiforov \cite{kani} initiated the study of Tur\'an-type problems for the $p$-spectral radius. This is defined as $h(G)=\max\{2\sum_{uv\in E(G)} x_ux_v: \, x_1,\dots,x_n\in\mathbb{R},\, |x_1|^p+\dots+|x_n|^p=1\}$. Li and Peng \cite{lipe} showed that $g_h$ is strongly $k$-ESS.

\textbf{Higher order spectral radius.} The $t$-clique tensor of a graph $G$ is an order $t$ dimension $n$ tensor, with entries $a_{i_1i_2\dotsi_i}=1/(t-1)!$ if $v_{i_1},\dots,v_{i_t}$ form a clique in $G$, and 0 otherwise. Let $h(G)$ denote the spectral radius of this tensor. Lu, Zhou and Bu \cite{lzb} showed that $g_h$ is strongly $(t+1)$-ESS.

\textbf{Local density.} Let $h(G)=h_\alpha(G)$ denote the smallest number of edges spanned by $\alpha n$ vertices of $G$. Keevash and Sudakov \cite{kesu} showed that if $1-1/2(k-1)^2\le \alpha\le 1$, then $g_h$ is strongly $k$-ESS.

\textbf{Small perturbations and combinations of $k$-ESS functions.} Clearly, if we add or multiply strongly $k$-ESS functions, we obtain new strongly $k$-ESS functions. An example where such functions have been studied is counting multiple graphs at the same time. Also, adding to a weakly or strongly $k$-ESS function $h(G)$, a function that is $o(h(G))$ results in another weakly or strongly $k$-ESS function. For example, counting stars $S_r$ and $\sum_{v\in V(G)}d(v)^r$ differs by a constant factor and a negligible additive term, thus one being weakly $k$-ESS implies the same for another.

\smallskip

Let us continue with some $k$-ESS-stable functions.

\textbf{Counting edges.} The well-known Erd\H os-Simonovits stability \cite{erd1,erd2,simi} means that $\ex(n,F)$ is $k$-ESS-stable.

\textbf{Counting (weakly) $F$-Tur\'an-stable graphs.}
The first stability result concerning $\ex(n,H,F)$ is due to Ma and Qiu \cite{mq}, who showed that counting cliques $K_t$ is $k$-ESS-stable if $k>t$. Several other results followed, and in fact by now in most cases when we know that counting $H$ is $k$-ESS, we also know that it is $k$-ESS-stable. Highlights include paths \cite{hhl}, complete $t$-partite graphs with $t<k$, and every graph if $k$ is large enough \cite{gerham}. Several other results can be found in \cite{gerb2}.

\textbf{Spectral radius.} Nikiforov \cite{niki2} showed that the spectral radius is strongly $k$-ESS-stable.

\textbf{$p$-spectral radius.} Li and Peng \cite{lipe} showed that $g_h$ is strongly $k$-ESS-stable.

\smallskip

Now let us look at an example of a strongly $k$-ESS-sigma function.

\textbf{Counting large complete balanced $(k-1)$-partite graphs.} Gerbner \cite{gerb3} showed that if $H$ is the complete $(k-1)$-partite graph $K_{a,\dots,a}$ and $a$ is large enough, then counting $H$ is strongly $k$-ESS-sigma.

Let us list some $k$-ESS-supersat functions.

\textbf{Counting subgraphs.} It was shown in \cite{ES83} that $\ex(n,F)$ is $k$-ESS-supersat. It was extended to every subgraph of chromatic number at most $k$ by Halfpap and Palmer \cite{halpal}.

\smallskip

Let us continue with listing some suitable partitions. It is clear that we obtain results for several partitions, in particular for several forbidden graphs with extra structure. Let us list some examples where $\cF$ is defined by some objects that have been studied before.

\textbf{Edge-ordered, vertex-ordered, cyclically ordered graphs.} These were the main examples of graphs with extra structure in \cite{Coreg,CoregRazb}. An interested reader may find more details \cite{bkv,GMNPTV,pachtar,tard}.

\textbf{Forbidden induced family of graphs.} This is another important example from \cite{CoregRazb}. This corresponds to hereditary partitions.

\textbf{Rainbow Tur\'an.} Keevash, Mubayi, Sudakov and Verstra{\"e}te \cite{kmsv} introduced the following problem. What is the maximum number of edges in an $n$-vertex graph that has a proper edge-coloring without a rainbow copy of $F$? Here rainbow copy of $F$ means that each edge gets a distinct color. Counting other subgraphs in this setting was initiated in \cite{gmmp}.

Keevash, Mubayi, Sudakov and Verstra{\"e}te \cite{kmsv} showed that the abstract chromatic number of $F$ is $\chi(F)$ by showing that any proper edge-coloring of the complete $k$-partite graph $K_{k^3t^3,\dots,k^3t^3}$ contains a rainbow copy of the complete $k$-partite graph $K_{t,\dots,t}$. Here we extend this by showing that $\sigma$ does not increase either.

\begin{lemma} 
    Let $\chi(F)=k$, $\sigma(F)=t$ and $n$ be sufficiently large. Then any proper edge-coloring of the complete $k$-partite graph $K_{t,n,\dots,n}$ contains a rainbow copy of $F$.
\end{lemma}

\begin{proof}
    Consider a properly edge-colored $K_{t,n,\dots,n}$ with parts $|X_1|=t$ and $X_2,\dots,X_k$. We let the color classes of $F$ be $|Y_1|=t$ and $Y_2,\dots,Y_k$. We will embed the sets $Y_i$ into $X_i$ greedily to obtain a rainbow copy of $F$. First we embed $Y_1$ into $X_1$ arbitrarily. After embedding $Y_1,\dots,Y_i$, we will embed $Y_{i+1}$. At this point we have embedded less than $|V(F)|$ vertices, thus there are less than $\binom{|V(F)|}{2}$ colors used on the already embedded edges. For each of the already embedded vertices $u$, we remove each vertex $v$ from $X_{i+1}$ if $uv$ is of a color already used. As there is at most one neighbor of $u$ in each color, less than $|V(F)|\binom{|V(F)|}{2}$ vertices were deleted from $X_{i+1}$, thus we can pick $|Y_{i+1}|$ other vertices if $n$ is sufficiently large. We can complete the embedding, thus the proof is complete.
\end{proof}

\section{An exact result for balanced graph parameters}\label{harom}

So far, we have not made any assumption on the graph parameter $h$ itself. We can prove another exact result if $h$ is balanced in the following sense. We think of the increase of $h$ when adding an edge to a graph as the contribution of that edge. We are interested in $h$ where we are given an upper bound on the order of magnitude of the contribution of every edge. Furthermore, when the graph is closer to a complete multipartite graph, then we also have the same lower bound on the order of magnitude of the contribution of every edge, i.e., the contributions have the same order of magnitude.

\begin{defn}
    We say that a graph parameter $h$ is \textit{balanced} for a positive integer $k$ if the following properties hold for some constant $a>0$:
    \begin{enumerate}[label=(\alph*)]
        \item For any graph $G$ and any non-edge $e$ of $G$, if $G'$ is obtained from $G$ by adding $e$, then $h(G')=h(G)+O(n^{a})$.
        \item If $G''$ is obtained from $G$ by adding a new vertex $u$ and joining $u$ to all the neighbors of an arbitrary vertex $v$, then $h(G'')=h(G)+O(n^{a+1})$.
        \item For any $c>0$ there is $n_0$ such that if $n\ge n_0$ and $G$ is a complete $(k-1)$-partite $n$-vertex graph with each part of order at least $cn$, then $h(G')=h(G)+\Theta(n^{a})$ and $h(G'')=h(G)+\Theta(n^{a+1})$.
        \item If $G$ is a complete $(k-1)$-partite $n$-vertex graph with each part of order at least $cn$ and $G'''$ is obtained from $G$ by deleting $x$ edges, then $h(G''')=h(G)-\Theta(xn^a)$.
    \end{enumerate}
\end{defn}

Recall that if the abstract chromatic number of $(\cA,\cF)$ is $k$, then each complete $(k-1)$-partite graph with each part of order at least $n$ is in $\cA$, for every sufficiently large $n$.
\begin{defn}
    A partition $(\cA,\cF)$ is \textit{edge-critical} if it is suitable with abstract chromatic number $k$ but for large enough $n$, no $G\in \cA$ contains $T^+(n,k-1)$ as a subgraph, where $T^+(n,k-1)$ is obtained from $T(n,k-1)$ by adding an edge into one of the smallest parts.  
\end{defn}

One particular example is when $\cF=\cF(F)$ for a $k$-chromatic graph $F$ with a color-critical edge, i.e., an edge whose deletion decreases the chromatic number. We will show in Proposition \ref{prop4.5} that the rainbow Tur\'an problem for a $k$-chromatic graph $F$ with a color-critical edge gives an edge-critical partition. However, in general, if we take some extra structure on $F$, it is unclear whether the corresponding partition is edge-critical. In fact, we are not aware of any such example.

Simonovits \cite{sim2} proved that if $F$ has a color-critical edge and $n$ is sufficiently large, then $\ex(n,F)=|E(T(n,\chi(F)-1))|$. We can extend this result to our setting in the balanced case if we also have stability.

\begin{thm}
    Let $g=g_h$ be a Tur\'an-type function that is weakly $k$-ESS-stable such that $h$ is balanced for $k$. Let $(\cA,\cF)$ be an edge-critical partition with abstract chromatic number $k$. Then $g(n,(\cA,\cF))=h(T)$ for some complete $(k-1)$-partite graph $T$ on $n$ vertices where $n$ is sufficiently large. 
\end{thm}

We follow the proof of Theorem 1.5 in \cite{gerb2}, which deals with counting copies of some graphs. We remark that in that theorem the assumption is slightly more general than being edge-critical. It is likely that a similar generalization would also hold in our setting.

\begin{proof} We pick $\varepsilon>0$ sufficiently small depending on $h$ and $\cA$. We also pick a sufficiently large $m$ depending on $\cA$ such that no graph in $\cA$ contains $T^+(m,k-1)$ as a subgraph. We pick $n$ to be sufficiently large with respect to $h$, $\cA$, $\varepsilon$ and $m$.

Let $G$ be an $n$-vertex graph in $\cA$ with $h(G)=g(n,(\cA,\cF))$. By the weakly $k$-ESS-stable property, $G$ can be turned into a complete $(k-1)$-partite graph $T$ by adding and/or deleting $o(n^2)$ edges. We pick $T$ so that we need to add and/or delete the smallest number of edges this way. Let $V_i$ be the $i$-th part of $T$ with $|V_1|\le |V_2|\le\dots\le |V_{k-1}|$. Observe that if $v\in V_i$ has $d$ neighbors in $V_i$ in $G$, then $v$ has at least $d$ neighbors in every $V_j$ in $G$, otherwise we could move it to $V_j$ and obtain another complete $(k-1)$-partite graph instead of $T$ that can be obtained from $G$ by adding and/or deleting a smaller number of edges.

\begin{clm}\label{ght}
    $|V_1|\ge \varepsilon n$ if $n$ is sufficiently large.
\end{clm}  

\begin{proof}[Proof of Claim]
Assume not and let $\ell$ be the largest integer such that $|V_\ell|<\varepsilon n$. Let us pick arbitrary vertex-disjoint subsets $U_i\subset V_{k-1}$ with $|U_i|=\lfloor n/(k-2)^2\rfloor$ for each $i\le \ell$. We move each $U_i$ form $V_{k-1}$ to $V_i$ to obtain another complete $(k-1)$-partite $n$-vertex graph $T'$. In other words, we delete each edge between $U_i$ and $V_i$ for every $i$, and then add each edge $uv$ with $u\in U_i$, $v\in V_{k-1}\setminus U_i$, for every $i$. Let us compare $h(T)$ and $h(T')$. We deleted $o(n^2)$ edges, which decreases $h$ by $o(n^{a+2})$. 

Now we consider the edges added in two steps. Let us pick a set $U_i'\subset U_i$ with $|U_i'|=\varepsilon n$ for each $i$. First, we add the edges connecting vertices in $U_i'$ to vertices in $V_{k-1}\setminus U_i$. Let $G_0$ denote the graph we obtain by deleting the vertices in $U_i\setminus U_i'$ for every $i$. Therefore $G_0$ is a complete $(k-1)$-partite graph on at least $n/2$ vertices with each part of order at least $\varepsilon n$. Then we add linearly many additional vertices, so that each vertex creates a complete $(k-1)$-partite graph, thus we can apply the definition of the balanced graph parameters to adding the next vertex. Therefore, each vertex increases $h$ by $\Theta(n^{a+1})$, which implies that altogether $h$ is increased by $\Theta(n^{a+2})$.

Thus we have obtained that $h(T')=h(T)+\Theta(n^{a+2})$. We also have $h(T)=(1+o(1))h(G)$. Observe that $h(T)=O(n^{a+2})$, since we can get $T$ from the empty graph by adding $O(n^2)$ edges. These imply that $h(T')=h(T)+\Theta(n^{a+2})>h(G)$, a contradiction.
\end{proof}

Let us return to the proof of the theorem and let $E$ denote the set of edges in $G$ that are not in $T$, i.e., the edges inside some $V_i$. Let $r(u)$ denote the number of edges incident to $u$ in $T$ that are not in $G$, i.e., the edges between parts that are missing from $G$. Then by the definition of $T$, we have $|E|=o(n^2)$ and $\sum_{u\in V(G)} r(u)=o(n^2)$. Let $A$ denote the set of vertices with $r(u)=o(n)$, then $|V(G)\setminus A|=o(n)$. Let $A_i=V_i\cap A$, then by Claim \ref{ght}, $|A_i|=\Omega(n)$. 

Let $B_i$ denote the set of vertices in $V_i$ with $\Omega(n)$ neighbors inside $V_i$.
 
\begin{clm}
   Any $u\in A_i$ has no neighbor in $A_i\cup B_i$.
\end{clm}    

\begin{proof}[Proof of Claim]
Assume that $uv\in E(G)$ with $v\in A_i\cup B_i$. We will show that $G$ contains $T^+(m,k-1)$, contradicting the definition of $m$. Let $U_j$ denote the $j$-th part of $T^+(m,k-1)$, and assume without loss of generality that we added the extra edge inside the $i$-th part. Then we embed the extra edge into $uv$ arbitrarily. We embed the other vertices of $U_i$ into $A_i$ arbitrarily. Next we embed some $U_j$ with $j\neq i$. Observe that $v$ has a set $W_j$ of $\Omega(n)$ neighbors in $A_j$, and only $o(n)$ vertices of $W_j$ are not adjacent to any given vertex of $A_i$. Therefore, with the exception of at most $(m-1)o(n)=o(n)$ vertices, the vertices of $W_j$ are in the common neighborhood of the already picked vertices. We pick $m$ of those vertices in $W_j$, and embed $U_j$ into those vertices.

We continue similarly, embedding the parts of $T^+(m,k-1)$ one by one. When we embed a part $U_\ell$, we pick $m$ vertices from the common neighborhood inside $A_\ell$ of the already embedded vertices. We have embedded 1 vertex into $v\in A_i\cup B_i$ and each other vertex (at most $m-1$ vertices) into some $A_j$ with $j\neq k$. Therefore, out of the $\Omega(n)$ neighbors of $v$ in $A_\ell$, only $o(n)$ vertices are not in the common neighborhood of the already picked vertices. This shows that indeed we can complete the embedding and obtain a contradiction.
\end{proof} 

Let us return to the proof of the theorem. 
The above claim implies that $B_i=\emptyset$, since a vertex in $B_i$ has at most $|V_i\setminus A_i|=o(n)$ neighbors inside $V_i$.
Let $X$ denote a smallest set of vertices inside $V(G)\setminus A$ such that each edge of $G$ inside parts is incident to at least one vertex of $X$. Then $\sum_{u\in X}r(u)=\Omega(n|X|)$. On the other hand, there are $o(n|X|)$ edges incident to $X$ inside $V_i$ because $B_i=\emptyset$. Let $G'$ denote the graph we obtain from $T$ by deleting the edges that are in $T$ but not in $G$. Then by the definition of balanced graph parameters, $h(G)=h(T)-\Omega(n^{a+1}|X|)$. We obtain $G$ from $G'$ by adding $o(n|X|)$ edges inside the parts, thus $h(G)=h(G)+o(n^{a+1}|X|)<h(T)$ if $|X|\neq 0$, a contradiction completing the proof.
\end{proof}

It is easy to see that several graph parameters mentioned earlier are balanced. We remark that our definition of balancedness was chosen to satisfy the requirements of each step of the proof. It would be interesting to find a simpler, yet similarly general definition of balancedness such that the above theorem holds.

Let us now show an edge-critical partition. Recall the rainbow Tur\'an problem from the previous section.

\begin{prop}\label{prop4.5}
    Let $F$ be a $k$-chromatic graph with a color-critical edge. Let $\cA$ denote the family of graphs that have a proper coloring without a rainbow copy of $F$, and $\cF$ denote the family of other graphs. Then $(\cA,\cF)$ is an edge-critical partition.
\end{prop}

\begin{proof}
Let us assume that $G$ contains $T^+(n,k-1)$ for $n$ sufficiently large and let $A_i$ be the $i$-th part of $T^+(n,k-1)$, with the extra edge $uv$ in $A_1$. Since $F$ has a color-critical edge, there is a $(k-1)$-partition of $F$ into $B_1,\dots, B_{k-1}$ such that the only edge inside parts is the edge $u'v'$ inside $B_1$.

We will embed $F$ into a properly colored $T^+(n,k-1)$ in a rainbow way, obtaining a contradiction.
First we map $u'$ into $u$ and $v'$ into $v$. Then we embed the rest of $B_1$ to the rest of $A_1$ arbitrarily. Afterward, we will embed the rest of the vertices of $B_2$ into $A_2$, then $B_3$ into $A_3$, and so on. We pick the order of the vertices inside the parts arbitrarily.

When we embed a vertex $w'\in B_i$, we have to pick a vertex $w$ of $A_i$ such that the edges connecting $w$ to the already embedded less than $|V(F)|$ vertices have color distinct from each of the at most $|E(F)|$ colors used earlier in the embedding. For each already embedded vertex $z$ and each already used color $c$, there is at most one vertex of $A_i$ that is joined by an edge of color $c$ to $z$. This shows that there are at most $|V(F)||E(F)|$ forbidden vertices in $A_i$, thus we can pick one where we embed $w$. We continue this way till we embed $F$, a contradiction completing the proof.
\end{proof}

\section{Concluding remarks}\label{negy}

There are several other Erdős-Stone-Simonovits-type theorems we can obtain by modifying our definitions a little bit. We say that $g_h$ is \textit{robust} if the following holds. If $G$ contains $T(n,k-1)$ and $G'$ is obtained from $G$ by adding and/or deleting $o(n^2)$ edges, then $h(G')=(1+o(1))h(G)$. Counting subgraphs of chromatic number at most $k-1$ clearly has this property. 

Let us say that a partition is \textit{suitable for robust Tur\'an-type functions} if there is a $k$ such that the following two properties hold. For sufficiently large $n$, for each complete $(k-1)$-partite graph $T$ with each part of order at least $n$, there is a graph in $\cA$ that can be obtained from $T$ by adding and/or deleting $o(n^2)$ edges. For every $G\in\cA$, we can delete $o(n^2)$ edges from $G$ to obtain a graph $G'$ that does not contain $T(n,k)$ as a subgraph, or if $K_n\in\cA$ for sufficiently large $n$. 

It is easy to see that some of the arguments in this paper extend to this situation. For example, if the the lower bound in Theorem 1.1 is obtained by a complete $(k-1)$-partite graph $T$, then we can find a graph $G\in\cA$ that is close to $T$ by the the suitability for robustness, and $h(G)$ is close to $h(T)$ by the robustness.

This can be applied to \textbf{regular Turán problems}.
Let $\cA$ consist of regular $F$-free graphs and $\cF$ consist of each other graphs. The study of Turán problems for regular graphs was initiated in \cite{cvk,catu}, where it was shown that for $k\ge 4$ and any $n$, there is a $(k-1)$-partite regular graph with $(1-o(1))|E(T(n,k-1))|$ edges. This implies that $(\cA,\cF)$ is suitable for robust Tur\'an-type functions. Note that counting other subgraphs in regular $F$-free graphs was studied in \cite{gerham2}.

Another example is the \textbf{shadow graph of Berge-$F$-free hypergraphs}. We omit the definitions here and only mention here that an Erd\H os-Stone-Simonovits-type theorem was proved for such graphs in \cite{luwang}, and the proof was by showing that the second property of the above definition holds (the first one holds trivially). Therefore, these graphs define a suitable partition for robust Tur\'an-type functions. Note that counting other subgraphs in such shadow graphs was initiated in \cite{tres}.

\smallskip

Instead of $(1+o(1))h(T)$, we can aim to obtain different bounds. An example is counting graphs $H$ on $h$ vertices, where this bound is of the form $\ex(n,H,F)=\cN(H,T)+o(n^h)$. It is known that in several cases $o(n^h)$ can be replaced by $O(n^{h-\varepsilon})$ for some $\varepsilon=\varepsilon(H)>0$; this implies $g(n,(\cA,\cF))=\cN(H,T)+O(n^{h-\varepsilon})$, which is analogous to Theorem \ref{main}.

Another example is \textbf{counting $n$-vertex $F$-free graphs}. A theorem of Erd\H os, Frankl and R{\"o}dl \cite{efr} states that there are $2^{(1+o(1))\ex(n,F)}$ distinct labeled $F$-free graph on $n$ vertices, if $F$ is not bipartite. This implies that if $(\cA,\cF)$ is a monotone partition with abstract chromatic number $k$, then there are $2^{(1+o(1))|(E(T(n,k-1))|}$ distinct $n$-vertex labeled graphs in $\cA$. The lower bound is obtained by the subgraphs of $T(n,k-1)$, while the upper bound comes from the above-mentioned theorem of Erd\H os, Frankl and R\H odl, since some $k$-chromatic graph is forbidden. 

Note that if we consider graphs with extra structures, this result only states how many underlying graphs there are, and does not say anything about how many orderings of the vertices of such a graph avoid the forbidden subgraphs. However, for the extra structures mentioned in this paper, the difference is negligible. For example, for an edge-ordered graph $F$ with abstract chromatic number $k>2$, we obtain that there are at most $2^{(1+o(1))|(E(T(n,k-1))|}\binom{n}{2}!=2^{(1+o(1))|(E(T(n,k-1))|}$ distinct $F$-free edge-ordered $n$-vertex labeled graphs.

In the case of \textbf{Turán problems for oriented graphs}, a parameter called \textit{compressibility} plays the role of the chromatic number in the analogue of the Erd\H os-Stone-Simonovits theorem \cite{vala}. It can be a direction of future research to examine whether some of our results extend to that setting.

\end{document}